\documentclass[english,reqno,12pt]{amsart}
\usepackage[english]{babel}
\usepackage[utf8]{inputenc}

\usepackage{amsmath}
\usepackage{amssymb}
\usepackage{amsthm}
\usepackage[pdftex]{graphicx}
\usepackage[pdftex,colorlinks,citecolor=blue,unicode]{hyperref}

\usepackage{enumitem}  
\usepackage{bbm}
\usepackage{cleveref}
\usepackage{multicol}

\makeatletter
\def\@marginparreset{%
  \reset@font
  \tiny
  \linespread{0.9}
  \@marginparfont
  \@setminipage
}
\def\@marginparfont{\normalfont\itshape\color{red}\raggedleft}
\makeatother

\crefname{theorem}{Theorem}{Theorems}
\crefname{lemma}{Lemma}{Lemmas}
\crefname{remark}{Remark}{Remarks}
\crefname{prop}{Proposition}{Propositions}
\crefname{defn}{Definition}{Definitions}
\crefname{corollary}{Corollary}{Corollaries}
\crefname{conjecture}{Conjecture}{Conjectures}
\crefname{chapter}{Chapter}{Chapters}
\crefname{section}{Section}{Sections}
\crefname{figure}{Figure}{Figures}
\crefname{example}{Example}{Examples}

\theoremstyle{plain}
\newtheorem{theorem}{Theorem}[section]
\newtheorem{lemma}[theorem]{Lemma}
\newtheorem{corollary}[theorem]{Corollary}
\newtheorem{prop}[theorem]{Proposition}

\theoremstyle{definition}
\newtheorem{defn}[theorem]{Definition}
\newtheorem{example}[theorem]{Example}
\theoremstyle{remark}
\newtheorem{remark}[theorem]{Remark}

\newcommand{\cF}{\mathcal F}

\newcommand{\cT}{\mathcal T}

\renewcommand{\subset}{\subseteq}

\renewcommand{\emptyset}{\varnothing}
\newcommand{\emp}{\emptyset}
\renewcommand{\leq}{\leqslant}
\renewcommand{\geq}{\geqslant}
\newcommand{\ds}{\displaystyle}

\DeclareMathOperator{\PP}{\mathbb{P}} 
\DeclareMathOperator{\EX}{\mathbb{E}} 
\DeclareMathOperator{\EXbn}{\mathbb{E}\!}
\DeclareMathOperator{\IND}{\mathbbm{1}} 
\DeclareMathOperator{\VAR}{\mathbb{V}\!\mathrm{a}\mathrm{r}}

\newcommand{\dt}{\textrm{d}t}
\newcommand{\du}{\textrm{d}u}
\newcommand{\e}{\textrm{e}}

\newcommand{\ZZ}{\mathbb{Z}}
\newcommand{\RR}{\mathbb{R}}

\DeclareMathOperator{\imag}{Im}

\newcommand{\boolf}{\colon \{-1, 1\}^n \to}

\DeclareMathOperator{\totinf}{Inf}
\DeclareMathOperator{\infl}{Inf}

\DeclareMathOperator{\cinfl}{CInf}
\DeclareMathOperator{\jinfl}{JInf}

\newcommand{\poch}{\partial}
\newcommand{\bW}{\mathbf W}
\DeclareMathOperator{\maxinf}{MaxInf}
\newcommand{\ff}{\hat{f}}
\renewcommand{\gg}{\hat{g}}
\newcommand{\HH}[1]{\widehat{H^{#1}}}
\newcommand{\eldwa}{L^2(\{-1,1\}^n, \mu_n)}

\newcommand{\xdopl}{x^{i\to +1}}
\newcommand{\xdomi}{x^{i\to -1}}

\newcommand{\bi}{\mathbf{i}}
\newcommand{\bj}{\mathbf{j}}
\newcommand{\bk}{\mathbf{k}}
\newcommand{\bx}{\mathbf{x}}

\renewcommand{\c}{\textbf{C1}}
\newcommand{\cc}{\textbf{C2}}
\newcommand{\ccc}{\textbf{C3}}

\title{KKL theorem for the influence of a set of variables}

\author[T. Przyby\l owski]{Tomasz Przyby\l owski}
\address{Mathematical Institute, Andrew Wiles Building, Observatory Quarter, University of Oxford, Oxford OX2 6GG, United Kingdom}
\address{Institute of Mathematics, University of Warsaw, ul. Banacha 2, 02-097 Warsaw, Poland}
\email{tomasz.przybylowski@kellogg.ox.ac.uk}


\begin{document}

\begin{abstract}
Consider a Boolean function $f$ on the $n$-dimensional hypercube, and a set of variables (indexed by) $S \subset \{1,2,\ldots,n\}.$ The coalition influence of the variables $S$ on a function $f$ is the probability that after a random assignment of variables not in $S$, the value of $f$ is undetermined. In this paper, we study a complementary notion, which we call the joint influence: the probability that, after a random assignment of variables not in $S$, the value of $f$ is dependent on all variables in $S$.



We show that for an arbitrary fixed $d$, every Boolean function~$f$ on~$n$ variables admits a~$d$-set of joint influence at least $\tfrac{1}{10} \bW^{\geq d}(f) (\frac{\log n}{n})^d$, where $\bW^{\geq d}(f)$ is the Fourier weight of $f$ at degrees at least $d$. This result is a direct generalisation of the Kahn--Kalai--Linial theorem. Further, we give an example demonstrating essential sharpness of the above bound. In our study of the joint influence we consider another notion of multi-bit influence recently introduced by Tal.
\end{abstract}

\maketitle

\section{Introduction}

Consider $i \in [n] = \{1,2,\ldots,n\}$ and $f \boolf \{-1,1\}$. Throughout, we consider the uniform measure $\mu_n$ on $\{-1, 1\}^n$. Define the \emph{influence} of bit $i$ on $f$ by 
\[\infl_i(f) := \underset{\bx \sim \mu_n}{\PP}( f(\bx) \neq f(\bx^{\oplus i})), \]
where $x^{\oplus i} = (x_1, \ldots, x_{i-1}, -x_i, x_{i+1}, \ldots, x_n)$.

The influence quantifies how often a given bit can sway the outcome of the function. It is a well-studied notion central to the field of Boolean analysis. Its history and various properties can be found in the textbook of Ryan O'Donnell \cite{O'D14}. A celebrated result, due to Kahn, Kalai and Linial \cite{KKL88}, says that every Boolean function admits a very influential bit. 


\begin{theorem}[KKL theorem]
There is a universal constant $C_1 > 0$ such that for every $f \boolf \{-1,1\}$ there is $i \in [n]$ such that \[\infl_i(f) \geq C_1 \VAR(f) \frac{\ln n}{n},\]
where $\VAR (f) = \underset{\bx \sim \mu_n}{\VAR} (f(\bx)).$
\end{theorem}

This is the best possible bound up to constant $C_1$ as certified by the \emph{tribe functions} introduced by Ben-Or and Linial~\cite{BL85}.

Several generalizations of the definition of the single-bit influence accounting for multiple bits have been considered, see \cite{biswassarkar} for discussion. The one probably most studied is the coalition influence, which we will now define.

In what follows, for brevity, we associate a variable $x_i \in\{-1,1\}$ with its index $i$. Let $\bi \subset [n]$ be a set of variables. The coalition influence $\cinfl_\bi(f)$ of $\bi$ on $f$ is defined as the probability that assigning values to the variables not in $\bi$ at random, the value of $f$ is undetermined. There has been a long line of research studying this type of influence, for the main results see \cite{KKL88}, \cite{AL93}, \cite{BKKKL92}.


The focus of our paper is a counterpart notion to the coalition influence. We call it the \emph{joint influence} $\jinfl_\bi(f)$ of $\bi$ on $f$ and define it as the probability that assigning values to the variables not in $\bi$ at random, the value of $f$ is dependent on \emph{all} of the variables in $\bi$.

More formally, say that $i \in \bi$ is \emph{$\bi$-pivotal} for $f$ on input $x$ if there is an input $y$ agreeing with $x$ on bits $[n] \setminus \bi$ such that $f(y) \neq f(y^{\oplus i})$. Then for non-empty $\bi$ define
\begin{align*}
    \jinfl_\bi(f) &:= \underset{\bx \sim \mu_n}{\PP}(\text{every }i\in \bi \text{ is } \bi \text{-pivotal for }f\text{ on input } \bx).
\end{align*}

Note that using $\bi$-pivotality, the coalition influence satisfies \[\cinfl_\bi(f) = \underset{\bx \sim \mu_n}{\PP}(\text{some }i\in \bi \text{ is } \bi \text{-pivotal for }f\text{ on input } \bx).\]
In particular, $\cinfl$ can be thought of as a \emph{union} of influences, whereas $\jinfl$ as an \emph{intersection}, and hence the name. As a result, it follows that $\cinfl_\bi(f) \geq \jinfl_\bi(f)$.

Let us mention yet another generalization of the notion of single bit influence, purely analytic one, as it will turn out useful in studying the joint influence. For $\bi \subseteq [n]$ we define the \emph{Walsh functions} $\chi_\bi \boolf \RR$~by $\chi_\bi(x) := \prod_{j \in \bi} x_j$ and $\chi_\emptyset(x) := 1$. The system $\{\chi_\bi\colon \bi \subseteq [n]\}$ is an orthonormal basis of $\eldwa$, which yields that every function $f\boolf \RR$ has a unique \emph{Fourier--Walsh expansion}
$f = \sum_{\bi \subseteq [n]} \ff(\bi) \chi_\bi$,
where $\ff(\bi) \in \RR$. We define the \emph{T-influence} $\infl_\bi(f)$ by
\begin{equation*}
\infl_\bi(f) := \sum_{\substack{\ \bj \subseteq [n]\!\colon \\ \bj \supseteq \bi}} \ff(\bj)^2.
\end{equation*}
It has been introduced by Tal in \cite{tal17} in his study of the Fourier mass at high levels in $\bold{AC^0}$ Boolean circuits. However, it has been independently considered by Tanguy in \cite{T20} for two-element $\bi$'s and in this case further studied by Oleszkiewicz in \cite{Ol21}. Often, when clear from the context, we will refer to it simply as the \emph{influence}. Sometimes, we will say that $\infl_\bi(f)$ is a \emph{$d$-degree} influence, where $d = |\bi|$.

The T-influence, although interesting on its own, will open to us the Fourier methods to bound the joint influence via inequality $\jinfl_\bi(f) \geq \infl_\bi(f)$, proven in \cref{derinotzero}. Further, note that in the case of $\bi = \{i\}$ all the three generalization of the single-bit influence are equal to $\infl_i(f)$ (with the equality $\infl_{\{i\}}(f) = \infl_i(f)$ being folklore).


\subsection{T-influence results}
Let us say a few words on what relevant to us has been proved on the T-influence.

The case of $|\bi| = 2$ has been considered by Tanguy in \cite{T20}, where he obtained a result in the spirit of the KKL theorem. Namely, he proved that every Boolean $f$ either admits a pair of bits with T-influence at least $C \VAR(f) (\frac{\ln n}{n})^2$, or there is a single bit with influence at least $C (\VAR(f)/n)^\theta$, for some universal constants $C>0$ and $\theta \in (1/2,1)$.

Later, Oleszkiewicz \cite{Ol21} extended some of Tanguy's bounds and proved that an appropriate upper bound on influences of all pairs of bits of a Boolean~$f$ implies that~$f$ is close to some affine Boolean function. 

\begin{theorem}[Oleszkiewicz]\label{oleszkiewicz}
There exists a universal constant $C_2>0$ with the following property. Let $n \geq 2$. Assume that $f \boolf \{-1,1\}$ satisfies $\infl_{\{i,j\}}(f) \leq \alpha \left(\frac{\ln n}{n}\right)^2$ for all $1 \leq i < j \leq n$ for some $\alpha > 0$. Then either
\begin{equation*}
\ff(\emptyset)^2 \geq 1-C_2\alpha, \quad \text{or there exists exactly one} \ i \in [n] \ \text{such that} \ \ff(\{i\})^2 \geq 1-C_2\alpha \frac{\ln n}{n}.
\end{equation*}
\end{theorem}

As Oleszkiewicz notes, the above theorem might be viewed as both a variation of the KKL theorem and another celebrated result due to Friedgut, Kalai and Naor \cite{FKN02}. For the purpose of the background of Oleszkiewicz's and our results, let us recall a generalisation of the FKN theorem due to Kindler and Safra \cite{KS02} (the FKN theorem is recovered by taking $d=1$).

\begin{defn}
Let $d \geq 0$ be an integer. We say that $g\boolf \{-1,1\}$ is of \emph{$d$-degree} if, for every $\bi \subset [n]$ such that $|\bi| > d$, we have $\gg(\bi) = 0$. Equivalently, $g$ is of $d$-degree if its degree as a Fourier--Walsh polynomial is at most $d$. 
\end{defn}

\begin{theorem}[Kindler--Safra] \label{kindlersafra}
Let $d$ be a positive integer. If $f\boolf \{-1,1\}$ satisfies $\bW^{\geq d+1}(f) \leq \alpha$, then there exists $g\boolf\{-1,1\}$ of $d$-degree such that
\begin{equation*}
\|f-g\|_2^2 \leq C_{KS}\alpha,
\end{equation*}
where $C_{KS}$ is a constant depending only on $d$ and $\bW^{\geq d+1}(f) = \sum_{\bi \subset [n]\colon\! |\bi| \geq d+1} \ff(\bi)^2$.
\end{theorem}

In the course of the proof of \cref{oleszkiewicz}, Oleszkiewicz shows that every Boolean function $f$ admits a pair of distinct bits $i$,~$j$ such that $\infl_{\{i,j\}}(f) \geq C_O \bW^{\geq2}(f) (\frac{\ln n}{n})^2$. Hence the assumption in \cref{oleszkiewicz} of small pair influences yields that $\bW^{\geq 2}(f)$ is small, so in particular implies the assumption of \cref{kindlersafra} for $d=1$. However, the stronger assumption in \cref{oleszkiewicz} yields the stronger consequence of extra $\ln n / n$ factor.

\subsection{Main results}

Inspired by the above, we prove the following theorem.

\begin{theorem} \label{Main Theorem}
Let $n \geq d \geq 1$ be integers. For every $f\boolf \{-1,1\}$ there exists $\bi \subset [n]$ of size $d$ so that
\begin{equation*}
\jinfl_\bi(f) \geq \infl_\bi(f) \geq \frac{1}{10} \bW^{\geq d}(f) \left(\frac{\ln n}{n}\right)^d,
\end{equation*}
where $\bW^{\geq d}(f) = \sum_{\bi \subset [n]\colon\! |\bi| \geq d} \ff(\bi)^2$.
\end{theorem}

Observe that for $d=1$ we recover the KKL theorem, and the case $d=2$ is the above-mentioned consequence of the proof of \cref{oleszkiewicz}. We prove \cref{Main Theorem} in \cref{sec:kklg} by generalising the representation of $\bW^{\geq 2}(f)$ and the hypercontractivity lemma from~\cite{Ol21}. Recall that the first inequality holds actually for every $\bi$ thanks to \cref{derinotzero}.

In \cref{sec:ex} we prove that \cref{Main Theorem} is asymptotically sharp in the following sense.

\begin{theorem}
For every $d \geq 2$ there is a constant $C^d_H$ depending only on $d$, and a family of functions $H_n^d$ such that for every $\bi \subset [n]$ of size $d$ we have
\[ \infl_\bi(H_n^d) \leq \jinfl_\bi(H_n^d) \leq C^d_H \bW^{\geq d}(H_n^d) \left(\frac{\log_2n}{n}\right)^d. \]
\end{theorem}

The functions $H_n^d$, which we call the \emph{$d$-hypertribes}, are a variation on the classical tribe functions due to Ben-Or and Linial \cite{BL85} demonstrating the sharpness of the KKL theorem. As was pointed out in \cite{Ol21}, the standard tribe function cannot serve as a sharp example in our case. E.g. for $d=2$ influences of pairs of bits from the same tribe are larger than the influence of a typical pair. To account for this, one needs to increase the size of tribes and allow voters to belong to multiple tribes in a controlled way. These adjustments complicate the structure of the function, and the explanation why it serves as a sharp example becomes far from trivial. The construction of the $d$-hypertribes was suggested to us by Peter Keevash in a private communication.

Further, in \cref{sec:prefkn,sec:fkn} we prove the following $d$-degree generalisation of \cref{oleszkiewicz}.

\begin{theorem} \label{FKN-type theorem}
Let $n > d \geq 1$. There are positive constants $C_3$, $C_4$ depending only on $d$ such that if $f \boolf \{-1,1\}$ satisfies
\begin{equation} \label{assumpfkn}
\infl_\bi(f) \leq \alpha \left(\frac{\ln n}{n}\right)^{d+1}
\end{equation}
for all $\bi\subset [n]$ of size $d+1$ for some $\alpha \in (0,C_3)$, then there exists $g \boolf \{-1,1\}$ of degree $d$ such that
\begin{equation*}\label{fknstat}
|\ff(\bj) - \hat{g}(\bj)| \leq C_4 \alpha \left(\frac{\ln n}{n}\right)^{|\bj|}
\end{equation*}
for every $\bj \subset [n]$ of size at most $d$.
\end{theorem}

For $d=1$ we recover \cref{oleszkiewicz}. Indeed, the only $1$-degree Boolean functions are the constants $-1$, $+1$, single-bit functions $\chi_{\{i\}} \equiv x_i$ and their negations. Depending which of these $g$ is, from \cref{fknstat} we can recover the relevant alternative of \cref{oleszkiewicz}.

We will see in \cref{sec:basic} that $\infl_\bi(f) = \|\poch_\bi f\|_2^2$, where $\poch_\bi f$ is a (discrete) $d$-order partial derivative of~$f$. From this point of view, \cref{FKN-type theorem} is a Boolean variant of a theorem from real analysis stating roughly that if all the $(d+1)$-order partial derivatives of $f\colon [0,1]^n \to \RR$ are small, then $f$ is close to some polynomial of degree at most $d$.

Similarly to \cref{oleszkiewicz}, \cref{FKN-type theorem} might be viewed as a variation on the Kindler--Safra \cref{kindlersafra}. And exactly like there, the assumption \eqref{assumpfkn} implies the Kindler and Safra assumption: we have $\bW^{\geq d+1}(f) \leq 10\alpha.$ 
However, \cref{FKN-type theorem} tells us that this stronger assumption yields a better approximation. 

Further, if we additionally assume that $f$ has small coefficients of low frequencies, then we obtain a strictly better bound on the distance $\|f-g\|_2^2$ compared to the Kindler--Safra Theorem. We prove this at the end of \cref{sec:fkn}.

\begin{corollary}\label{ourcorr}
Let $n > d \geq 1$. Assume that $f \boolf \{-1,1\}$ satisfies \eqref{assumpfkn} and that for some $l \in [d]$ we have $|\ff(\bi)| \leq 2^{-d}$ for all $\bi \subset [n]$ of size at most $l-1$. Then there exists a function $g \boolf \{-1,1\}$ of $d$-degree such that
\begin{equation*} \label{kwantyzacja}
\|f-g\|_2^2 \leq C_5\alpha \left(\frac{\ln n}{n}\right)^l,
\end{equation*}
where $C_{5}$ is a constant depending only on $d$.
\end{corollary}

\subsection{Acknowledgements}
The author was supported by the Additional Funding Programme for Mathematical Sciences, delivered by EPSRC (EP/V521917/1) and the Heilbronn Institute for Mathematical Research. The majority of this work was written as a master's thesis at the University of Warsaw. I would like to thank my supervisor Krzysztof Oleszkiewicz for help and for an interesting topic to research, and to Christina Goldschmidt, Peter Keevash and Oliver Riordan for valuable suggestions and discussions.

\section{Notations}\label{subsec:nota} 

Throughout the paper, $n$ is an arbitrary positive integer (possibly further constrained). Let $d \in \ZZ_{\geq 0}$. We use the standard notation $[n]=\{1,2,\ldots,n\}$ and the less standard notation
\[ [n]_d := \{ A \subseteq [n]\colon |A| = d\}, \quad [n]_{<d} := \{ A \subseteq [n]\colon |A| < d\}. \]
We define $[n]_{\leq d}$, $[n]_{>d}$ and $[n]_{\geq d}$ in an analogous manner.

We write $\imag(f)$ for the image of a function $f$, and refer to functions with $\imag(f) \subset \{-1,1\}$ as \emph{Boolean} functions. Plain letters $i$, $j$, $k$ usually denote indices of bits of the input, i.e.\ elements of $[n]$, whereas bold letters $\bi$, $\bj$, $\bk$ refer to index sets, i.e.\ subsets of $[n]$. For brevity, we will associate a variable $x_i \in\{-1,1\}$ with its index $i$, and so we will say e.g. \emph{variables~$\bi$} instead of \emph{variables $\{x_i\colon i \in \bi\}$}. Sometimes, for clarity, we omit brackets and write e.g. $\bi \setminus i$ instead of $\bi \setminus \{i\}$. The name \emph{$d$-set} will stand for a $d$ element subset of $[n]$.

For an input $x = (x_1,\ldots,x_n)$, $i\in[n]$ and $\varepsilon \in \{-1,1\}$ we write 
\begin{align*}
    x^{\oplus i} &= (x_1,\ldots,x_{i-1},-x_i, x_{i+1},\ldots,x_n), \\
    x^{i \to \varepsilon} &= (x_1,\ldots,x_{i-1},\varepsilon, x_{i+1},\ldots,x_n),
\end{align*}
and further, for $\bi \subset [n]$ and $y \in \{-1,1\}^\bi$ define $x^{\bi \to y} \in \{-1,1\}^n$ by
\begin{align*}
    x^{\bi \to y}_j = \begin{cases}
        x_j, \text{ if } j \in [n] \setminus \bi, \\
        y_j, \text{ if } j \in \bi.
    \end{cases}    
\end{align*}
%

\section{Basics of harmonic analysis} \label{sec:basic}

In this section we introduce the standard language of the Boolean analysis. The basic Fourier--Walsh notions are recalled and we repeat the definition of influences in a broader context. Finally, we state some results regarding \emph{hypercontractivity}, which we will use extensively in the proofs of both main theorems. 
Many of these notions (with exception of the multi-bit influences) are explained in greater detail in O'Donnell's book \cite{O'D14}.

Let $\mu_n = (\frac{1}{2}\delta_{-1} + \frac{1}{2}\delta_1)^{\otimes n}$ be the uniform measure on $\{-1,1\}^n$. Let us consider the linear space $\eldwa$ of all functions $f \boolf \RR$ equipped with the standard inner product and norms given by
\begin{gather*}
\langle f, g \rangle := \frac{1}{2^n}\sum_{x \in \{-1,1\}^n} f(x)g(x) = \EX [f(\bx)g(\bx)] = \EX [fg] \quad \text{and} \\ 
\quad \|f\|_2 := \sqrt{\langle f, f\rangle} = \sqrt{\EX [f(\bx)^2]} = \sqrt{\EXbn f^2}. 
\end{gather*}
The probabilities $\PP$ and expectations $\EX$ are always taken with respect to the measure $\mu_n$. The bold letter $\bx$ will then denote a uniform random variable with distribution $\mu_n$. Note that, due to the product structure of $\mu_n$, the set of single bits $\bx_i$ of $\bx$ is a family of independent and identically distributed random variables, each having distribution $\frac{1}{2}\delta_{-1} + \frac{1}{2}\delta_1$. Often, when it is clear from the context, we omit the letter $\bx$ in probabilities or expectations.

For $\bi \subseteq [n]$ we define the \emph{Walsh functions} $\chi_\bi \boolf \RR$~by $\chi_\bi(x) := \prod_{j \in \bi} x_j$ and $\chi_\emptyset(x) := 1$. The system $\{\chi_\bi\colon \bi \subseteq [n]\}$ is an orthonormal basis of $\eldwa$, which yields that every function $f\boolf \RR$ has a unique \emph{Fourier--Walsh expansion}
\begin{equation*}
f = \sum_{\bi \subseteq [n]} \ff(\bi) \chi_\bi,
\end{equation*}
where we call $\ff(\bi) := \langle f, \chi_\bi \rangle \in \RR$ a \emph{Fourier--Walsh coefficient}.

In this setting, every $f, g \boolf \RR$ satisfy the \emph{Plancherel} and \emph{Parseval identities}
\[\langle f, g \rangle = \sum_{\bk \subseteq [n]} \ff(\bk)\hat{g}(\bk) \quad \text{and} \quad \|f\|_2^2 = \sum_{\bk \subseteq [n]} \ff(\bk)^2. \]

Further, note that we have $\EX f = \langle f, 1\rangle = \langle f, \chi_\emptyset\rangle = \ff(\emp)$. Moreover, if $f$ is Boolean (so $\imag{f} \subset \{-1,1\}$) we have $\|f\|_2^2 = \sum_{\bi \subseteq [n]} \ff(\bi)^2 = 1$ and so in this case $|\ff(\bi)| \leq 1$ for every $\bi \subseteq [n]$.

Also, for $f \boolf \RR$ we define
\begin{equation*}
\bW^{\geq d}(f) := \sum_{\bi \in [n]_{\geq d}} \ff(\bi)^2 \quad \text{and}\quad \bW^{=d}(f) := \sum_{\bi \in [n]_{d}} \ff(\bi)^2.
\end{equation*}

We will say that $f$ is of \emph{$d$-degree} if for every $\bi \in [n]_{>d}$ we have $\ff(\bi) = 0$. Equivalently, $f$ is of $d$-degree if its degree as a Fourier--Walsh polynomial is at most $d$.

\subsection{Derivative}

In the following section we introduce the notion of the discrete partial derivatives.
We remark that observations regarding partial derivatives (\cref{poch_multi} and \cref{deri_sum_formula} below) were already made in \cite{tal17}.

 Let $f \boolf \RR$, $i \in [n]$. Define the \emph{$i$-th partial derivative} $\poch_i f \boolf \RR$ by\footnote{Our notation $\poch_i$ is consistent with \cite{Ol21}, but inconsistent with \cite{O'D14, tal17} where the symbol $D_i$ is used instead.}
\begin{equation*} \label{defipoch}
\poch_i f (x) := \frac{f(\xdopl)-f(\xdomi)}{2}.
\end{equation*}
By linearity, one can easily get a useful Fourier--Walsh representation
\begin{equation} \label{pochodna}
\poch_i f = \sum_{\substack{\ \bj \subseteq [n]\!\colon \\ \bj \ni i}} \ff(\bj) \chi_{\bj \setminus \{i\}}.
\end{equation}

The derivative operators may be composed. The order of composition does not matter due to the Fourier--Walsh representation (\ref{pochodna}). In consequence, we can give the following general definition.

\begin{defn}\label{poch_multi} Let $f\boolf \RR$ and $\bi  = \{i_1, i_2, \ldots, i_k\} \subseteq [n]$. We define the \emph{$\bi$-th partial derivative $\poch_\bi f \boolf \RR$} by
\begin{equation*}
\poch_\bi f := \poch_{i_1} \circ \poch_{i_2} \circ \ldots \circ \poch_{i_k} f = \sum_{\substack{\ \bj \subseteq [n]\!\colon\\ \bj \supseteq \bi}} \ff(\bj) \chi_{\bj \setminus \bi} \qquad \text{and} \qquad \poch_\emptyset f := f.
\end{equation*}
\end{defn}

A trivial induction proves that the definition of the single-bit derivative extends to a multi-bit counterpart in the following way.

\begin{prop} \label{deri_sum_formula}
Let $f \boolf \RR$ and $\bi \subset [n]$ be non-empty. Then
\begin{equation} \label{even_sum}
\poch_\bi f (x) = \frac{1}{2^{|\bi|}} \sum_{y \in \{-1,1\}^\bi} \Big(f(x^{\bi \to y}) \cdot \prod_{j \in \bi} y_j \Big).
\end{equation}
\end{prop}

\begin{corollary} \label{der_int_mult}
Let $f \boolf \{-1,1\}$ and $\bi \subset [n]$ be non-empty. Then \[\imag(\poch_\bi f) \subset \ZZ / 2^{|\bi|-1} \cap [-1,1].\]
\end{corollary}

\begin{proof}
The definition of the single-bit derivative, or \cref{deri_sum_formula}, easily yields $\imag(\poch_\bi f) \subset [-1,1]$. For the other containment, observe that after application of \cref{deri_sum_formula}, all we need to prove is that the sum in \eqref{even_sum} is an even integer. This is indeed the case as it is a sum of $2^{|\bi|}$ odd integers.
\end{proof}

\subsection{Influence}\label{sec:influence}
In the following section we will connect the influence definitions from the introduction with the notion of derivative, and prove some useful bounds between them.

Firstly, let us connect the definition of influence of a single-bit with the Fourier-Walsh coefficients. Observe that $f(x) \neq f(x^{\oplus i})$ occurs if and only if $\poch_i f(x) \neq 0$, which occurs if and only if $|\poch_i f(x)| = 1$. Therefore, by Parseval's identity
\begin{equation*}
    \infl_i(f) = \PP( \poch_i f(\bx) \neq 0) = \EX [(\poch_i f(\bx))^2] = \|\poch_i f\|_2^2 = \sum_{\substack{\ \bj \subseteq [n]\!\colon \\ \bj \ni i}} \ff(\bj)^2.
\end{equation*}
In particular, as we promised in the introduction, $\infl_i(f)$ is equal to the T-influence $\infl_{\{i\}}(f)$.

The definition of the multi-bit derivative yields an analogous representation of the T-influence. Let $f \boolf \RR$ and $\bi \subseteq [n]$. Again, by Parseval's identity we have
\begin{equation*}
\infl_\bi(f) = \sum_{\substack{\ \bj \subseteq [n]\!\colon \\ \bj \supseteq \bi}} \ff(\bj)^2 = \|\poch_\bi f\|_2^2.
\end{equation*}

\begin{prop}\label{infl_lower_bd}
Let $f \boolf \{-1,1\}$, $\bi \subseteq [n]$ and $r = |\bi| > 0$. Then
\begin{enumerate}[label=(\roman*)]
\item $\PP(\poch_\bi f \neq 0) \geq \infl_\bi(f) \geq \frac{1}{2^{2r-2}} \PP(\poch_\bi f \neq 0)$;
\item $\infl_\bi(f) \geq \frac{1}{2^{r-1}} |\ff(\bi)|$.
\end{enumerate}
\end{prop}
\begin{proof}
By \cref{der_int_mult} we have $(\poch_\bi f)^2 \in [\IND_{\{\poch_\bi f \neq 0\}} / 2^{2r-2}, \IND_{\{\poch_\bi f \neq 0\}}]$, so taking expectation yields (i).

For (ii), again, by \cref{der_int_mult} we obtain
\[\infl_\bi(f) = \EX [ |\poch_\bi f|^2] \geq \EX\Big[\frac{1}{2^{r-1}} |\poch_\bi f|\Big] \geq \frac{1}{2^{r-1}} | \EX [\poch_\bi f ] | = \frac{1}{2^{r-1}} |\ff(\bi)|. \]
\end{proof}

\begin{prop} \label{derinotzero}
Let $f\boolf \{-1,1\}$ and a non-empty $\bi \subset [n]$ satisfy $\poch_\bi f(x) \neq 0$ for some $x \in \{-1,1\}^n$. Then every $i \in \bi$ is $\bi$-pivotal for $f$ on $x$. In particular, $\jinfl_\bi(f) \geq \infl_\bi(f)$.
\end{prop}

\begin{proof}
Take an arbitrary $i \in \bi$, and denote by $y^+ = y^{i\to +1}$ and $y^- = y^{i\to -1}$. By \cref{deri_sum_formula} we have
\begin{align*}
\poch_\bi f (x) = (\poch_i (\poch_{\bi \setminus i} f)) (x) &= 
\frac{\poch_{\bi \setminus i} f (x^{i\to +1}) - \poch_{\bi \setminus i} f (x^{i\to -1})}{2} \\
&= \frac{1}{2^{|\bi|}} \sum_{y \in \{-1,1\}^{\bi \setminus i}} \Big( f(x^{\bi \to y^+}) - f(x^{\bi \to y^-}) \Big) \prod_{j \in \bi \setminus i} y_j.
\end{align*}
Since $\poch_\bi f(x) \neq 0$, there exists a non-zero summand in the above sum. As a result, $i$ is $\bi$-pivotal for $f$ on $x$, which holds for every $i \in \bi$.

For the inequality part, the above implies $\jinfl_{\bi}(f) \geq \PP(\poch_\bi f \neq 0)$ and so by \cref{infl_lower_bd} we get $\jinfl_{\bi}(f) \geq \infl_\bi(f)$.
\end{proof}

\begin{remark}
    For $\bi$ of size $1$ or $2$ one can check that the first part of the proposition is actually an equivalence, in which case $\jinfl_\bi(f) = \PP(\poch_\bi f \neq 0)$.
\end{remark}

Let us introduce for brevity two more pieces of notation.
\begin{defn}
Let $f \boolf \RR$ and $d \leq n$ be a positive integer. We define
\begin{equation*}
\totinf(f) := \sum_{i \in [n]} \infl_{i}(f) \quad \text{and} \quad \maxinf_d (f) := \max_{\bi \in [n]_d} \infl_\bi(f).
\end{equation*}
\end{defn}

\subsection{Hypercontraction and the log-Sobolev inequality}

The key element of the proofs of the KKL, FKN and Kindler--Safra theorems is the celebrated hypercontractivity theorem by Bonami~\cite{B70}, presented below. It is also the central ingredient of our proofs of \cref{Main Theorem,FKN-type theorem}.

\begin{defn}
Let $f\boolf \RR$ and $t\geq 0$. We define $P_t f \boolf \RR$ by
\begin{equation*}
P_t f := \sum_{\bi \subseteq [n]} \e^{-|\bi|t} \ff(\bi) \chi_\bi.
\end{equation*}
The family of operators $(P_t)_{t\geq 0}$ is often called \emph{the heat semigroup} on the Boolean cube.
\end{defn}

\begin{theorem}[\cite{B70}]\label{hiperro}
Let $f \boolf \RR$ and $t\geq 0$. Then
\begin{equation*}
\|P_t f \|_2 \leq \|f\|_{1+\e^{-2t}}.
\end{equation*}
\end{theorem}

The second important fact we make use of is \cref{logsob}---a weak functional version of the edge isoperimetric inequality---which is a consequence of the classical log-Sobolev inequality. We will use it in proving \cref{FKN-type theorem}.

\begin{lemma} \label{logsob}
If $f\boolf \{0,1\}$, then 
$\totinf(f) \geq \frac{1}{2} \EXbn f \cdot \ln(1 / \EXbn f).$
\end{lemma}

For the proof and background of the lemma we refer the reader to Section 5 of \cite{Ol21}.

\section{Proof of \cref{Main Theorem}} \label{sec:kklg}

In this section we prove \cref{Main Theorem}. Our proof proceeds via two propositions. The first, Proposition~\ref{kklprop1}, provides an integral representation for the higher order variance $\bW^{\geq d}(f)$ via the heat semi-group operator $P_t$. The second, Proposition~\ref{kklprop2}, uses hypercontractivity and gives a bound on integrals of the representation of the higher order variance in terms of influences. These propositions are generalisations of Lemmas 3.1 and 3.2 from~\cite{Ol21}. The proofs presented proceed in a fairly analogous manner to those in \cite{Ol21}.
\begin{prop} \label{kklprop1}
Let $n \geq d \geq 1$ be integers. For every $f\boolf \RR$ we have
\begin{equation*}
\bW^{\geq d}(f) = \sum_{\bi \in [n]_{\geq d}} \ff(\bi)^2 = 2d \sum_{\bi \in [n]_d} \int_0^\infty (\e^{2t}-1)^{d-1} \e^{-2dt} \|P_t \poch_\bi f \|_2^2 \dt.
\end{equation*}
\end{prop}

\begin{proof}
Since for $\bi \in [n]_d$ we have
\begin{gather*}
P_t \poch_\bi f = P_t \Bigg(\!\!\sum_{\substack{\ \bk \subseteq [n]\!\colon\\ \bi \subseteq \bk}} \ff(\bk) \chi_{\bk \setminus \bi}\Bigg) = \sum_{\substack{\ \bk \subseteq [n]\!\colon\\ \bi \subseteq \bk}} \e^{-(|\bk|-d)t} \ff(\bk) \chi_{\bk \setminus \bi},
\end{gather*}
by the Parseval identity we obtain
\begin{align*}
\sum_{\bi \in [n]_{d}} \e^{-2dt} \|P_t \poch_\bi f \|_2^2 &= \sum_{\bi \in [n]_d} \e^{-2dt} \sum_{\substack{\ \bk \subseteq [n]\!\colon\\ \bi \subseteq \bk}} \e^{-2(|\bk|-d)t} \ff(\bk)^2 = \sum_{\bi \in [n]_d} \sum_{\substack{\ \bk \subseteq [n]\!\colon\\ \bi \subseteq \bk}} \e^{-2|\bk|t} \ff(\bk)^2 =\\
&= \sum_{\bk \in [n]_{\geq d}} \sum_{\substack{\ \bi \in [n]_d\!\colon\\ \bi \subseteq \bk}} \e^{-2|\bk|t} \ff(\bk)^2 = \sum_{\bk \in [n]_{\geq d}} {|\bk| \choose d} \e^{-2|\bk|t} \ff(\bk)^2.
\end{align*}
Therefore,
\begin{align*}
\sum_{\bi \in [n]_d} \int_0^\infty (\e^{2t}-1)^{d-1} \e^{-2dt} \|P_t \poch_\bi f \|_2^2 \dt = \sum_{\bk \in [n]_{\geq d}} {|\bk| \choose d} \ff(\bk)^2 \int_0^\infty (\e^{2t}-1)^{d-1} \e^{-2|\bk|t} \dt.
\end{align*}
Substituting $u = \e^{-2t}$, $u \in (0,1]$, i.e.\ $t = -\frac{1}{2}\ln u$ with $\big|\frac{\dt}{\du}\big| = \frac{1}{2u}$, we get
\begin{align*}
\int_0^\infty (\e^{2t}-1)^{d-1} \e^{-2|\bk|t} \dt = \int_0^1 \Big(\frac{1}{u}-1\Big)^{d-1}  u^{|\bk|} \cdot \frac{\du}{2u} &= \frac{1}{2} \int_0^1 (1-u)^{d-1}u^{|\bk|-d} \du = \\ 
&= \frac{1}{2} \textrm{B}(d, |\bk|-d+1) = \frac{1}{2d \cdot {|\bk| \choose d}},
\end{align*}
which yields the statement.
\end{proof}

\begin{prop} \label{kklprop2}
Let $n$, $d$, $l$ be integers such that $d \geq l \geq 1$. Assume that $f\boolf \RR$ and some $\bi \subset [n]$ satisfy $\imag(\poch_\bi f) \subset \frac{1}{2^{d-1}}\ZZ$ and $\infl_\bi(f) \leq 1$. Then
\begin{gather*}
\int_0^\infty (\e^{2t}-1)^{l-1}\e^{-2lt} \| P_t \poch_\bi f \|^2 \dt \ \ \leq \ \ (l-1)! 4^{d-1} \cdot \frac{\infl_\bi(f)}{\ln^l(1 / \infl_\bi(f))}.
\end{gather*}
\end{prop}
\begin{remark}
When $\infl_\bi(f) \in \{0,1\}$ we let $1 \cdot (\ln 1/1)^{-l} = +\infty$ and $0 \cdot (\ln 1/0)^{-l} = 0$.
\end{remark}

\begin{remark}
The seemingly artificial assumption that $\imag(\poch_\bi f) \subset \tfrac{1}{2^{d-1}} \ZZ$ is satisfied already for $f$ Boolean and $\bi \in [n]_d$, due to \cref{der_int_mult}. The assumption takes such a form since it will also be used in the proof of \cref{FKN-type theorem}, where a slightly more general form is needed.
\end{remark}

\begin{proof}
If $\infl_\bi(f) = 1$, then the right-hand side equals $+\infty$. If $\infl_\bi(f) = 0$, then $\poch_\bi f \equiv 0$, so the integral is equal to $0$. Therefore, in both cases the statement is true. Henceforth, we will assume that $\infl_\bi(f) \in (0,1)$.
 
The Bonami hypercontractivity bound applied to the function $\poch_\bi f$ gives that
\begin{equation} \label{eqhyper}
\|P_t \poch_\bi f\|_2^2 \leq \| \poch_\bi f \|^2_{1+\e^{-2t}} = \left(\EX [|\poch_\bi f|^{1+\e^{-2t}}]\right)^{\frac{2}{1+\e^{-2t}}}.
\end{equation}
Observe that for every $v \in \imag(\poch_\bi f) \subseteq \frac{1}{2^{d-1}} \ZZ$, we have
\begin{equation*}
|v|^{1+\e^{-2t}} \leq 2^{(d-1)(1-\e^{-2t})} \cdot v^2.
\end{equation*}
This yields that
\begin{equation} \label{eq2hyper}
\EX [|\poch_\bi f|^{1+\e^{-2t}}] \ \leq \ 2^{(d-1)(1-\e^{-2t})} \EX[(\poch_\bi f)^2] \ = \ 2^{(d-1)(1-\e^{-2t})} \cdot \infl_\bi(f).
\end{equation}
Due to (\ref{eqhyper}) and (\ref{eq2hyper}), we obtain
\begin{equation} \label{boundzior}
\|P_t \poch_\bi f\|_2^2 \leq (2^{2(d-1)})^{\frac{1-\e^{-2t}}{1+\e^{-2t}}} \cdot \infl_\bi(f)^{1+\frac{1-\e^{-2t}}{1+\e^{-2t}}} = 2^{2(d-1)u} \cdot \infl_\bi(f)^{1+u},
\end{equation}
where we substituted $u = \frac{1-\e^{-2t}}{1+\e^{-2t}}$, i.e.\ $t = \frac{\ln(1+u)-\ln(1-u)}{2}$. As $t \in [0,\infty)$, we have $u \in [0,1)$ and $\frac{\dt}{\du} = \frac{1}{(1-u)(1+u)}$. Using bound the (\ref{boundzior}) we get
\begin{align*}
\int_0^\infty (\e^{2t}-1)^{l-1}\e^{-2lt} \| P_t \poch_\bi f \|^2 \dt \leq \int_0^1 \left(\frac{2u}{1-u}\right)^{l-1} \left(\frac{1-u}{1+u}\right)^l 2^{2(d-1)u}  \frac{\infl_\bi(f)^{1+u}}{(1+u)(1-u)} \du  \\
\leq 4^{d-1}\infl_\bi(f) \int_0^1 \left(\frac{2^u}{1+u}\right)^{l+1} u^{l-1} \infl_\bi(f)^u \du \leq 4^{d-1}\infl_\bi(f) \int_0^\infty u^{l-1} \infl_\bi(f)^u \du,
\end{align*}
where we used the bounds $2^{l-1+2(d-1)u} \leq 2^{2d-2+(l+1)u}$
and $2^u \leq 1+u$ for $u \in [0,1]$.

Finally,
\begin{align*}
\int_0^\infty u^{l-1} \infl_\bi(f)^u \du = \int_0^\infty u^{l-1} \e^{u \ln(\infl_\bi(f))} \du = \frac{1}{\ln^l(1 / \infl_\bi(f))} \int_0^\infty t^{l-1}\e^{-t} \dt = \frac{(l-1)!}{\ln^l(1 / \infl_\bi(f))},
\end{align*}
which finishes the proof.
\end{proof}

We are ready to prove \cref{Main Theorem}.

\begin{proof}[Proof of \cref{Main Theorem}]
For $n=1$ the statement is trivial. Therefore, let us assume that $n \geq 2$.

Observe that for $\bi \in [n]_d$ we have $\imag(\poch_\bi f) \subset \ZZ/2^{d-1}$. Hence, by \cref{kklprop1,kklprop2} for $l=d$ and all $\bi \in [n]_d$ we obtain
\begin{equation*}
\bW^{\geq d}(f) = 2d \sum_{\bi \in [n]_d} \int_0^\infty (\e^{2t}-1)^{d-1} \e^{-2dt} \|P_t \poch_\bi f \|_2^2 \dt \leq d!2^{2d-1} \sum_{\bi \in [n]_d} \frac{\infl_\bi(f)}{\ln^d(1/\infl_\bi(f))}.
\end{equation*}
Since the function $t \mapsto t(\ln(1/t))^{-d}$ is increasing on $[0,1]$, we get
\begin{equation} \label{kklfinal}
\bW^{\geq d}(f) \leq d!2^{2d-1} {n \choose d} \frac{\maxinf_d(f)}{\ln^d(1/\maxinf_d(f))} \leq 2^{2d-1} n^d \frac{\maxinf_d(f)}{\ln^d(1/\maxinf_d(f))}.
\end{equation}
Assume that $\maxinf_d(f) \leq \frac{1}{n^{d/2}}$. The function $t \mapsto (\ln(1/t))^{-d}$ is increasing on $[0,1]$, so by~(\ref{kklfinal}) we obtain
\begin{equation*}
\bW^{\geq d}(f) \leq 2^{2d-1} n^d \cdot \frac{\maxinf_d(f)}{((d/2)\ln n)^d} = \frac{2^{3d-1}}{d^d} \cdot \left(\frac{n}{\ln n}\right)^d \maxinf_d(f) \leq 10 \left(\frac{n}{\ln n}\right)^d \maxinf_d(f),
\end{equation*}
and so the proof in this case is done.

On the other hand, if $\maxinf_d(f) > \frac{1}{n^{d/2}}$, then since $\ln n \leq \sqrt{n}$ for all $n \in \ZZ_+$, we have
\begin{equation*}
\maxinf_d(f) > \frac{1}{n^{d/2}} \geq \left(\frac{\ln n}{n}\right)^d \geq \bW^{\geq d}(f) \cdot \left(\frac{\ln n}{n}\right)^d,
\end{equation*}
as $\bW^{\geq d}(f) \leq \|f\|_2^2 = 1$, which finishes the proof.
\end{proof}

\section{The hypertribe example}\label{sec:ex}

In the following section, we show the essential sharpness of the lower-bound in \cref{Main Theorem}. More precisely, for every $d \geq 2$ we construct a family of \emph{hypertribe} functions $H_n^d$ satisfying $\maxinf_d(H_n^d) \leq C \bW^{\geq d}(H_n^d) (\frac{\log n}{n})^d$ for every $n \geq d$ and some constant $C=C(d) > 0$. The construction we present was suggested to us by Peter Keevash.

Let us elaborate on some quantities behind the construction and why its $1$-degree prototype does not work.

In the case of the KKL theorem, the sharpness is provided by the (standard) tribe functions due to Ben-Or and Linial~\cite{BL85}. For the construction, one splits $n$ voters (bits) into disjoint tribes of sizes roughly $\log (\frac{n}{\log n})$ and declares the outcome (i.e.\ output of the function) to be $+1$ if there is a tribe such that all its members voted $+1$, and to be $-1$, if there is no such tribe. 

The standard tribe functions, as noted in~\cite{Ol21}, do not give sharpness in \cref{Main Theorem} already for $d=2$. In this case, although the influences of pairs of voters from distinct tribes are of the (desired) order $(\tfrac{\log n}{n})^2$, the influences of pairs of voters from the same tribe are of order $\tfrac{\log n}{n}$. If one accounts for that, and changes the sizes of the tribes to $2\log (\frac{n}{\log n})$, the variance (and thus also $\bW^{\geq 2}$) drops down from constant to $O(\tfrac{\log n}{n})$ magnitude, which is again not enough to get sharpness. 

In order to make up for this drop, we need to increase the number of tribes. This comes at the cost of tribes no longer being disjoint, which in turn may increase the order of magnitude of the influences. To keep them of the desired small order, we need to guarantee that for every pair of voters (or, more generally, every $d$-tuple) they are not simultaneously members of too many tribes. Fortunately, the Kuzjurin result (the second part of Theorem 1 in~\cite{K95}) regarding existence of certain asymptotically good packings guarantees the existence of the desired functions.

We start by stating \cref{kuzju}, a special case of Theorem 1 from~\cite{K95}. Based on this, in \cref{hyperfn} we construct the hypertribe functions $H_n^d$. In \cref{properties}, we state some basic quantitative observations regarding the structure of~$H_n^d$. In \cref{prop111,prop333,prop444} we prove bounds on the influences of $H_n^d$ and the probabilities of the outcomes $\pm1$. Finally, in \cref{hyperthm} we collect the proven bounds and show that the hypertribes indeed asymptotically attain the lower bound from \cref{Main Theorem}.

\begin{theorem}[\cite{K95}]\label{kuzju} Fix $d \geq 2$. Assume that $k(n)$ satisfies $k(n) \stackrel{n \to \infty}{\longrightarrow} \infty$ and $k(n) \leq c\sqrt{n}$ for some positive constant $c<1$ and sufficiently large $n$. Then there exists $N > 0$ such that for every $n > N$ there is a family $\mathcal{F}_n$ of $k(n)$-element subsets of~$[n]$ such that 
\begin{enumerate}
\item[($\ast$)] for at least half of the $d$-sets $\bi \in [n]_d$ there is a set $A_\bi \in \mathcal{F}_n$ such that $\bi \subset A_\bi$;
\item[($\ast\ast$)] for every $d$-set $\bi \in [n]_d$ there is at most one set $A \in \mathcal{F}_n$ such that $\bi \subset A$. 
\end{enumerate}
\end{theorem}

\noindent The family $\cF_n$ can be understood as a \emph{packing} of $d$-set-disjoint $k(n)$-sets. The above theorem asserts that this can be done in such a way that at least half of the $d$-sets are covered.

\begin{example}($d$-hypertribes) \label{hyperfn}
Fix $d \geq 2$. Let $k(n) = d \log_2 \frac{n}{\log_2 n}$, which is an integer when $n=2^{2^a}$ and $a \in \ZZ_+$. For sufficiently large $n$ we clearly have $k(n) \leq \frac{1}{2}\sqrt{n}$ and $k(n) \stackrel{n \to \infty}{\longrightarrow}~\infty$. Choose $n > N$ so that $k(n)$ is integer. Then, by Theorem~\ref{kuzju}, we get a family $\mathcal{F}_n$ of $k(n)$-element subsets of $[n]$ satisfying both ($\ast$) and ($\ast\ast$). 

Let us treat $+1$ and $-1$ as the Boolean \emph{true} and \emph{false}, respectively. We construct the $d$-hypertribe function $H_n^d \boolf \{-1,1\}$ as follows
\begin{equation*}
H_n^d(x) := \bigvee_{A \in \mathcal{F}_n} \bigwedge_{i \in A} x_i = \textrm{OR}\left(\textrm{AND}(x_1^1, \ldots, x_{k(n)}^1), \ldots, \textrm{AND}(x_1^t, \ldots, x_{k(n)}^t)\right),
\end{equation*}
where $t := |\cF_n|$ and $x_i^j \in A_j \in \mathcal{F}_n$. In other words, $H_n^d(x)$ is the outcome of (some special kind of) voting of $t$ tribes, each consisting of $k(n)$ members. If there is any tribe that unanimously votes for \emph{true}, then the outcome of $H_n^d$ is \emph{true}. Otherwise, the outcome of $H_n^d$ is \emph{false}. Tribes are not necessarily disjoint and each voter has to send the same vote in all the tribes to which they belong.

From now on we simply write $k$ and $H_n$ instead of $k(n)$ and $H_n^d$.
\end{example}

\begin{prop}\label{properties} The family $\cF_n$ constructed in \cref{hyperfn} satisfies the following properties.

\begin{enumerate}[label=(\textbf{P\arabic*})]

\item\label{p5} For every $\bi \in [n]_{\leq d}$,
\begin{equation*}
\big|\{A \in \mathcal{F}_n \colon \bi \subset A\}\big| \leq \left(\frac{n}{\log_2 n}\right)^{d-|\bi|};
\end{equation*}

\item\label{p3} There exists a constant $c_1 > 0$ depending only on $d$ such that
\[c_1 \leq t \cdot 2^{-k} \leq 1;\]

\item\label{p6} Every two distinct tribes $A$, $B \in \cF_n$ have $|A \cap B| \leq d-1$.
\end{enumerate}
\end{prop}

\begin{proof} \phantom{ }
\begin{enumerate}[label=(\textbf{P\arabic*})]
\item Fix $\bi \in [n]_{\leq d}$ and let $s = |\bi|$. There are $n-s \choose d-s$ $d$-sets in $[n]_d$ containing $\bi$. Every tribe $A \in \cF_n$ such that $\bi \subset A$ covers exactly $k-s \choose d-s$ $d$-sets containing $\bi$. Therefore, by $(\ast\ast)$,
\begin{equation*}
\big|\{A \in \mathcal{F}_n \colon \bi \subset A\}\big| \leq \frac{{n-s \choose d-s}}{{k-s \choose d-s}} \leq \frac{n^{d-s}}{k^{d-s}} \leq \left(\frac{n}{\log_2 n}\right)^{d-s}.
\end{equation*}

\item The upper bound follows from \ref{p5} applied to $\bi = \emptyset$. For the lower bound note that every tribe covers $k \choose d$ $d$-sets. Hence, by $(\ast)$,
\begin{equation*}
t \geq \frac{\frac12 {n \choose d}}{ {k \choose d}} \geq \frac{n^d}{2 \cdot k^d} \geq \frac{1}{2d^d} \left(\frac{n}{\log_2 n}\right)^d = \frac{1}{2d^d} \cdot 2^k,
\end{equation*}
and the statement follows with $c_1 = \tfrac1{2d^d}$.
\item This follows from $(\ast\ast)$. \qedhere
\end{enumerate}
\end{proof}

We now turn to proving specific bounds on $H_n$, which will give us the desired sharpness. For that let us introduce some convenient notation for speaking about the inputs of $H_n$.

\begin{defn} For $A \subset [n]$ we define the event $\cT(A) \subset \{-1,1\}^n$ by
\begin{equation*}
\cT(A) := \{x \in \{-1,1\}^n\colon \ \text{for every} \ i \in A \ \text{there is} \ x_i = 1\}.
\end{equation*}
\end{defn}

\noindent Clearly, $H_n(x) = 1$ if and only if $x \in \bigcup_{A\in \cF_n} \cT(A)$.

\begin{prop} \label{prop111}
There is a constant $c_2$ depending only on $d$ such that for every non-empty $\bi \subset [n]$ of size at most $d$,
\[ \jinfl_\bi(H_n) \leq c_2 \left(\frac{\log_2 n}{n}\right)^{|\bi|}. \]
\end{prop}

\begin{proof}
Consider $x\in \{-1,1\}^n$ such that every $i \in \bi$ is $\bi$-pivotal for $H_n$ on $x$. By the definition of the $\bi$-pivotality, for every $i$ there is $y \in \{-1,1\}^n$ such that $y_j = x_j$ for every $j \in [n] \setminus \bi$ and
\[ H_n(y^{i\to+1}) \neq H_n(y^{i\to-1}).\]
As $H_n$ is increasing, the above may be rewritten as
\[H_n(y^{i\to+1}) = 1 \quad \text{and} \quad 
H_n(y^{i\to-1}) = -1. \]
It follows that there is a tribe $A \in \cF_n$ (which may not be unique) such that $i \in A$ and $x \in \cT(A \setminus \bi)$.


Let $B_\bi$ denote the set of all partitions of $\bi$. Taking into account that the above property holds for all $i \in \bi$, we conclude that there is a partition $S \in B_\bi$ and distinct tribes $A_\bj$ for $\bj \in B_\bi$, so that $x \in \cT(A_\bj \setminus \bi)$. Hence,
\begin{align}
\jinfl_\bi(H_n) &\leq
\PP\Bigg( \bigcup_{S \in B_\bi} \bigcup_{\substack{\{A_\bj\}_{\bj \in S} \subset \cF_n\colon \\ \forall \bj \in S \colon \bj \subset A_\bj}} \bigcap_{\bj \in S} \cT(A_\bj \setminus \bi) \Bigg) \nonumber \\ &\leq 
\sum_{S \in B_\bi} \sum_{\substack{\{A_\bj\}_{\bj \in S} \subset \cF_n\colon \\ \forall \bj \in S \colon \bj \subset A_\bj}} \PP\bigg(\bigcap_{\bj \in S} \cT(A_\bj \setminus \bi) \bigg) \label{long_sum}.
\end{align}

As the number of partitions of $\bi$ is bounded by some function of $d$, by (\ref{long_sum}) it is enough to show that for every partition $S \in B_\bi$ we have
\begin{equation} \label{enough_to_show}
\sum_{\substack{\{A_\bj\}_{\bj \in S} \subset \cF_n\colon \\ \forall \bj \in S \colon \bj \subset A_\bj}} \PP\left(\bigcap_{\bj \in S} \cT(A_\bj \setminus \bi) \right) \leq \widetilde{c_2} \cdot \bigg(\frac{\log_2 n}{n}\bigg)^r.
\end{equation}
Clearly,
$\bigcap_{\bj \in S} \cT(A_\bj \setminus \bi) = \cT\big(\bigcup_{\bj \in S} A_\bj \setminus \bi\big)$. Further, as $|S| \leq |\bi| \leq d$ and $\ref{p6}$, by the inclusion-exclusion principle, 
\[\bigg| \bigcup_{\bj \in S} A_\bj \setminus \bi \bigg| \geq |S|k - {|S| \choose 2}(d-1) - d \geq |S|k - d^3,\] so that
\begin{equation} \label{enough_proof1}
\PP\bigg( \bigcap_{\bj \in S} \cT(A_\bj \setminus \bi) \bigg) 
\leq 2^{-|S|k + d^3} = 2^{d^3} \cdot \left(\frac{\log_2 n}{n}\right)^{d|S|}.
\end{equation}
We now prove that the number of terms in the sum in (\ref{enough_to_show}) is sufficiently small. Indeed, because of $\ref{p5}$, for every $\bj \in S$, 
\begin{equation*}
\big| \{A_\bj \in \cF_n\colon \bj \subset A_\bj \} \big| \leq \left( \frac{n}{\log_2 n} \right)^{d-|\bj|},
\end{equation*}
so that
\begin{equation} \label{enough_proof2}
\big| \{ \{A_\bj\}_{\bj \in S} \subset \cF_n\colon \forall \bj \in S \text{ there is } \bj \subset A_\bj \} \big| \leq \prod_{\bj \in S} \left(\frac{n}{\log_2 n}\right)^{d-|\bj|} = \left(\frac{n}{\log_2 n}\right)^{d|S|-r}.
\end{equation}
Finally, by (\ref{enough_proof1}) and (\ref{enough_proof2}),
\begin{align*}
\sum_{\substack{\{A_\bj\}_{\bj \in S} \subset \cF_n\colon \\ \forall \bj \in S \colon \bj \subset A_\bj}} \PP\bigg(\bigcap_{\bj \in S} \cT(A_\bj \setminus \bi) \bigg) \leq
\left(\frac{n}{\log_2 n}\right)^{d|S|-r} \cdot 2^{d^3} \left(\frac{\log_2 n}{n}\right)^{d|S|} = 2^{d^3} \left(\frac{\log_2 n}{n}\right)^r,
\end{align*}
which yields the statement with $c_2 = |B_\bi| \cdot 2^{d^3} \leq d^d2^{d^3}$.
\end{proof}

In the next two propositions, we prove that $H_n$ attains both values $\pm 1$ with at least positive constant probability. For the proof of \cref{prop444} we use the Harris inequality for decreasing sets, which we recall below. In \cref{prop333} the argument proceeds via a standard second moment method. Afterwards we finally prove \cref{hyperthm}.

\begin{defn}
We say that an event $S \subset \{-1,1\}^n$ is \emph{decreasing} if for every $x, y \in \{-1,1\}^n$ the conditions $x \in S$ and $x \geq y$ coordinate-wise imply that $y \in S$.
\end{defn}

\begin{theorem}[\cite{H60}]
If $S_1$, $S_2$, \ldots, $S_m \subset \{-1,1\}^n$ are all decreasing events, then
\begin{equation*}
\PP\bigg(\bigcap_{i \in [m]} S_i\bigg) \geq \prod_{i \in [m]} \PP(S_i).
\end{equation*}
\end{theorem}

\begin{prop} \label{prop444}
There is a universal constant $c_3>0$ such that $\PP(H_n=-1) \geq c_3$.
\end{prop}
\begin{proof}
Let $\mathcal{F}_n = \{A_i\colon i \in [t]\}$. Denote by $X^c$ the complement of the event $X \subset \{-1,1\}^n$. 

We have
$\{H_n=-1\} = \bigcap_{i \in [t]} \cT(A_i)^c$.
Further, for every $i\in[t]$, the event $\cT(A_i)^c$ is decreasing. Hence, by the Harris inequality for the events $\cT(A_1)^c$, \dots, $\cT(A_t)^c$,
\begin{equation*}
\PP(H_n=-1) = \PP\bigg(\bigcap_{i \in [t]} \cT(A_i)^c\bigg) \geq \prod_{i \in [t]} \PP\big(\cT(A_i)^c\big) = (1-2^{-k})^t \geq \e^{-t2^{-k+1}} \stackrel{\ref{p3}}{\geq} \e^{-2},
\end{equation*}
where we used the bound $1-x \geq \e^{-2x}$ which is valid for $x \in [0,\frac14]$.
\end{proof}

\begin{prop} \label{prop333}
There is a constant $c_4\! >\! 0$ depending only on $d$ such that $\PP(H_n=+1)\geq c_4$.
\end{prop}
\begin{proof}
Let $\mathcal{F}_n = \{A_i\colon i \in [t]\}$. Further, let $N = \sum_{i \in [t]} N_i$, where $N_i = \IND_{\cT(A_i)}$ are random variables.  

Observe that $\{H_n=+1 \} = \{N > 0 \}$. By the Paley-Zygmund inequality,
\begin{align}\label{paleyzygmund}
\PP(H_n=+1) &= \PP(N > 0) \geq \frac{(\EX N)^2}{\EX [N^2]} =
\frac{(t \EX N_1)^2}{\ds \sum_{1 \leq i \leq t} \EX [N_i^2] + 2 \sum_{ 1 \leq i < j \leq t} \EX [N_i N_j]}.
\end{align}
Clearly, $\EX N_1 = \EX [N_i^2] = \PP(\cT(A_i)) = 2^{-k}$. For the expectations of the products, note that \[\EX [N_iN_j] = \PP(\cT(A_i) \cap \cT(A_j)) = \PP(\cT(A_i \cup A_j)).\] By \ref{p6}, $|A_i \cup A_j| > 2k-d$ for distinct $i$, $j \in [t]$, so $\EX [N_iN_j] \leq 2^{-2k+d}$. Hence, (\ref{paleyzygmund}) yields
\[ \PP(H_n = +1) \geq \frac{t^2 2^{-2k}}{t2^{-k} + t(t-1) 2^{-2k+d} } \geq \frac{t2^{-k}}{1 + 2^d\cdot t2^{-k}}  \stackrel{\ref{p3}}{\geq} \frac{c_1}{1+2^dc_1}, \]
and the statement follows with $c_4 = \frac{c_1}{1+2^dc_1}$.
\end{proof}

\begin{theorem}\label{hyperthm}
For every $d \geq 2$ there is a constant $c_5$ depending only on $d$ such that
\[ \jinfl_\bi(H_n) \leq c_5 \bW^{\geq d}(H_n) \left(\frac{\log_2n}{n}\right)^d \]
for all $\bi \in [n]^d$. In particular, for every $d\geq 2$ \cref{Main Theorem} is asymptotically sharp.
\end{theorem}


\begin{proof}
By \cref{prop111} we have $\jinfl_\bi(H_n) \leq c_2\big(\frac{\log_2 n}{n}\big)^d$ for all $\bi \in [n]^d$. Hence, it is enough to show that there is a constant $c_6 > 0$ such that for sufficiently large $n$ we have $\bW^{\geq d}(H_n) > c_6$. 

Since $H_n$ is Boolean,
\begin{equation} \label{postac}
\bW^{\geq d} (H_n) = 1 - \HH{n}(\emp)^2 - \sum_{r = 1}^{d-1} \bW^{=r}(H_n).
\end{equation}
By \cref{prop444,prop333},
\begin{gather} \label{postac2}
1 - \HH{n}(\emp)^2 = 4 \PP(H_n = +1) \PP(H_n = -1) \geq 4c_3c_4.
\end{gather}
Further, for every $r \in [d-1]$, by \cref{infl_lower_bd},
\begin{align*}
\bW^{=r}(H_n) = \sum_{\bi \in [n]_r} \HH{n}(\bi)^2 \leq 2^{2d} \sum_{\bi \in [n]_r} (\infl_\bi(H_n))^2 \leq 2^{2d} n^r (\maxinf_r(H_n))^2 
\end{align*}
and by \cref{derinotzero,prop111} this may be further bounded above by
\begin{align*}
c_2^2 2^{2d}n^r \cdot \left(\frac{\log_2 n}{n}\right)^{2r} =
 c_2^2 2^{2d} \cdot \frac{(\log_2 n)^{2r}}{n^{r}},
\end{align*}
so $\sum_{r = 1}^{d-1} \bW^{=r}(H_n) < 2c_3c_4$ for sufficiently large $n$. Hence, by (\ref{postac}) and (\ref{postac2}), for sufficiently large $n$, we have $\bW^{\geq d}(H_n) > c_3c_4$, as desired.
\end{proof}

\section{The FKN-type theorem for the multi-bit influences}\label{sec:prefkn}

We now turn our attention to \cref{FKN-type theorem}. Roughly speaking, it says that a Boolean function which is close to a $d$-degree function in the sense of having small $(d+1)$-degree influences is necessarily close to a $d$-degree Boolean function. Such a statement might be seen as a variation of the FKN theorem~\cite{FKN02}, or its higher order generalization---the Kindler--Safra theorem \cite{KS02}---\cref{kindlersafra}. Indeed, the Kindler--Safra theorem ensures that a Boolean function which is close to a $d$-degree function in the sense of having small $(d+1)$-order variance $\bW^{\geq d+1}$ is necessarily close to a $d$-degree Boolean function. As it turns out, the assumption of having small $(d+1)$-degree influences is strictly stronger than having small $(d+1)$-order variance $\bW^{\geq d+1}$. However, this stronger assumption yields a better approximation than the one provided by the Kindler--Safra theorem.

Additionally, \cref{FKN-type theorem} is a direct generalization of Oleszkiewicz's \cref{oleszkiewicz}, which can be recovered by taking $d=1$. In his paper, \cite{Ol21}, Oleszkiewicz presents two separate proofs of \cref{oleszkiewicz}; one solely uses hypercontraction, the other the log-Sobolev inequality. Our proof of \cref{FKN-type theorem} combines the two approaches. It seems unlikely to us that just using one of them would be sufficient in this higher order case. 

In the following section we lay the foundation for the proof of \cref{FKN-type theorem}. First, in \cref{fknprop2} we prove an integral representation of $\infl_\bj(g) - \gg(\bj)^2$. In \cref{fknprop4,fknprop5} we prove some approximation results regarding influences and Fourier--Walsh coefficients. Further, in \cref{wsppocho} we prove that the Fourier--Walsh coefficients of a $d$-degree Boolean function lie in the $\ZZ/2^{d-1}$ lattice.

\subsection{Hypercontractive propositions}

In the following proposition we prove a higher order analogue of Lemma 3.3 from~\cite{Ol21}.

\begin{prop} \label{fknprop2}
Let $n > k \geq 0$ be integers. For every $g\boolf \RR$ and $\bj \in [n]_k$,
\begin{gather*}
\infl_\bj(g) - \gg(\bj)^2 = 2 \sum_{\substack{\ \bi \in [n]_{k+1}\!\colon\\ \bj \subseteq \bi}} \int_0^\infty \e^{-2t} \| P_t \poch_\bi g \|^2 \dt.
\end{gather*}
\end{prop} 

\begin{proof}
Let $\bj \in [n]_k$. Observe that $(\widehat{\poch_\bj g})(\emp) = \EX [\poch_\bj g] = \gg(\bj)$. Thus, by the definition of influence,
\begin{equation*}
\infl_\bj(g) - \gg(\bj)^2 = \|\poch_\bj g\|_2^2 - (\widehat{\poch_\bj g})(\emp)^2 = \bW^{\geq 1} (\poch_\bj g).
\end{equation*}
Further, \cref{kklprop1} applied to the function $f = \poch_\bj g$ and $d=1$ yields
\begin{equation*}
\bW^{\geq 1} (\poch_\bj g) = 2 \sum_{i \in [n]} \int_0^\infty \e^{-2t} \|P_t \poch_i (\poch_\bj g)\|_2^2 \dt = 2 \sum_{\substack{\ \bi \in [n]_{k+1}\!\colon\\ \bj \subseteq \bi}} \int_0^\infty \e^{-2t} \| P_t \poch_\bi g \|^2 \dt,
\end{equation*}
which finishes the proof.
\end{proof}

The next proposition relates how close is a Fourier--Walsh coefficient to the respective influence in terms of influences of one degree higher. We prove it via \cref{kklprop2} which uses hypercontraction.

\begin{prop} \label{fknprop4}
Let $n$, $d$, $k$ be integers such that $n \geq 2$, $d \geq 1$ and $n > k \geq 0$. Assume that $g \boolf \RR$ and $\beta \in [0,1]$ satisfy
\begin{gather*}
\maxinf_{k+1}(g) \leq \beta \left(\frac{\ln n}{n}\right)^{k+1} \quad \text{and for every} \ \bi \in [n]_{k+1} \ \text{there is} \ \imag(\poch_\bi g) \subseteq \ZZ/2^{d-1}.
\end{gather*}
Then for every $\bj \in [n]_k$,
\begin{gather*}
\infl_\bj(g) - \gg(\bj)^2 \leq 2^{2d} \beta \left(\frac{\ln n}{n}\right)^k.
\end{gather*}
\end{prop}

\begin{proof}
Let $\bi \in [n]_{k+1}$ be arbitrary. Then we have $\infl_\bi(g) \leq \maxinf_{k+1}(g)~\leq~1$. Therefore, applying Proposition \ref{kklprop2} to the function $f = g$, $l=1$ and $\bi$ we get
\begin{equation} \label{staryprop}
\int_0^\infty \e^{-2t} \| P_t \poch_\bi g \|^2 \dt \ \ \leq \ \  \frac{2^{2d-2} \infl_\bi(g)}{\ln(1/\infl_\bi(g))}.
\end{equation}

Fix $\bj \in [n]_k$. By Proposition \ref{fknprop2}, (\ref{staryprop}) and the fact that the function $t \mapsto t/\ln(1/t)$ is increasing on $[0,1]$, we obtain
\begin{align*}
\infl_\bj(g) - \gg(\bj)^2 &= 2 \sum_{\substack{ \ \bi \in [n]_{k+1} \!\colon \\ \bj \subseteq \bi}} \int_0^\infty \e^{-2t} \| P_t \poch_\bi g \|^2 \dt \leq 2^{2d-1} \sum_{\substack{ \ \bi \in [n]_{k+1} \!\colon \\ \bj \subseteq \bi}} \frac{\infl_\bi(g)}{\ln(1/\infl_\bi(g))} \\ 
&\leq 2^{2d-1}\frac{n\cdot \maxinf_{k+1}(g)}{\ln(1/\maxinf_{k+1}(g))} \leq 
2^{2d-1} \cdot \beta \left(\frac{\ln n}{n}\right)^k \cdot \frac{\ln n}{\ln\left(\frac{1}{\beta} \cdot \big(\frac{n}{\ln n}\big)^{k+1}\right)}.
\end{align*}

\noindent Observe that $n \geq 2$, so that $\sqrt{n} \geq \ln n$. Thus, since $\beta \leq 1$, we get
\begin{equation*}
2^{2d-1} \beta \left(\frac{\ln n}{n}\right)^k \cdot \frac{\ln n}{\ln\left(\frac{1}{\beta} \cdot \big(\frac{n}{\ln n}\big)^{k+1}\right)} \leq
2^{2d-1} \beta \left(\frac{\ln n}{n}\right)^k \cdot \frac{\ln n}{\ln \sqrt n} = 2^{2d} \beta \left(\frac{\ln n}{n}\right)^k. \qedhere
\end{equation*}
\end{proof}

\subsection{Log-Sobolev propositions}

In the following propositions we prove that a function that is close to a $d$-degree Boolean function, in some specific sense, has its Fourier--Walsh coefficients close to integer multiples of $1/2^{d-1}$. \cref{fknprop1} is a minor modification of Corollary~5.4 from \cite{Ol21}.

Let us introduce some convenient notation. 

\begin{defn}
Let $d \geq 1$ be an integer. We define $V_d = [-2^{2d}, 2^{2d}] \cap (\ZZ / 2^{d-1})$.
\end{defn}

\begin{prop} \label{fknprop1}
Let $d \geq 2$ be an integer. Assume $g \boolf V_d$ satisfies~$\totinf(g)<\tfrac{1}{2^{5d}}$. Then there exists $\eta \in \imag(g) \subseteq V_d$ such that
\begin{gather*}
\PP(g \neq \eta) \leq  \frac{2^{2d-1} \totinf(g)}{\ln\left(\frac{1}{2^{5d}\totinf(g)}\right)}.
\end{gather*}
\end{prop}

\begin{proof}
If $\totinf(g) = 0$, then $g$ is constant and in this case $\eta = g$ satisfies the statement.

Assume, henceforth, that $\totinf(g) > 0$. Note that 
\begin{equation*}
\sum_{\theta \in V_d} \PP(g \neq \theta) = \sum_{\theta \in V_d} (1 - \PP(g = \theta)) = |V_d| - 1 = 2^{3d}.
\end{equation*}
Therefore, we may pick $\eta \in V_d$ such that
\begin{gather}
\PP(g \neq \eta) \leq \frac{|V_d|-1}{|V_d|}. \label{ubound}
\end{gather}
Let $h = \IND_{ \{g \neq \eta\}}$. We claim that
\begin{gather}
\totinf(h) \leq 2^{2d-2} \cdot \totinf(g). \label{claim1}
\end{gather}
Indeed, note that
\begin{align*}
|\poch_i h (x)| = \frac{1}{2}\big|h(\xdopl) &- h(\xdomi)\big| = \frac{1}{2} \big| \IND_{ V_d \setminus \{\eta\} }(g(\xdopl)) - \IND_{ V_d \setminus \{\eta\} }(g(\xdomi)) \big| \\
&\leq \frac{2^{d-1}}{2} \cdot \big|g(\xdopl) - g(\xdomi)\big| = 2^{d-1} \cdot | \poch_i g(x)|,
\end{align*}
where we used the fact that $\min\{|v-w| \colon v,w \in V_d, v\neq w\} = 1/2^{d-1}$.
This inequality yields that for every $i \in [n]$ it holds that $\| \poch_i h \|_2^2 \leq 2^{2d-2} \cdot \| \poch_i g \|_2^2$. The claim follows from summing this relation over $i \in [n]$.

Observe that $\EXbn h = \PP(g \neq \eta)$. By claim (\ref{claim1}), the log-Sobolev Corollary \ref{logsob} applied to the $\{0,1\}$-valued function $h$, the monotonicity of $\ln$ with (\ref{ubound}) and the bound $\ln(1+x) \geq \tfrac{x}{2}$ for $x \in (0,1)$, we get
\begin{gather}
2^{2d-2} \totinf(g) \geq \totinf(h) \geq \frac{1}{2} \EXbn h \cdot \ln \left(\frac{1}{\EXbn h}\right) \geq \frac{1}{2} \EXbn h \cdot \ln \left(1+\frac{1}{|V_d|-1}\right) \geq \frac{\EXbn h}{2^{3d+2}}. \label{logsobuse}
\end{gather}
Inequality (\ref{logsobuse}) implies that $\EXbn h \leq 2^{5d} \totinf(g)$. Using this, we may improve the lower bound of $\ln(1/\EXbn h)$ in (\ref{logsobuse}) as follows:
\begin{gather}
2^{2d-2}\totinf(g) \geq \frac{1}{2} \EXbn h \cdot \ln \left(\frac{1}{\EXbn h}\right) \geq \frac{1}{2} \EXbn h \cdot \ln \left(\frac{1}{2^{5d}\totinf(g)}\right). \label{logsobuse2}
\end{gather}
The restriction $\totinf(g) < 2^{-5d}$ guarantees that $\ln\left(\frac{1}{2^{5d}\totinf(g)}\right) > 0$. Recalling that $\EXbn h = \PP(g \neq \eta)$ and rewriting (\ref{logsobuse2}) completes the proof.
\end{proof}

\begin{prop} \label{fknprop5}
Let $n$, $d$, $k$ be integers such that $n > k \geq 1$ and $d\geq 2$. Assume that $g \boolf \RR$ and $\beta \in [0,\tfrac{1}{2^{6d}})$ satisfy
\begin{equation*}
\maxinf_{k+1}(g) \leq \beta \left(\frac{\ln n}{n}\right)^{k+1} \quad \text{and for every} \ \bj \in [n]_k \ \text{we have} \ \imag(\poch_\bj g) \subseteq V_d.
\end{equation*}
Then for every $\bj \in [n]_k$ there exists $\eta_\bj \in \imag(\poch_\bj g) \subseteq V_d$ such that
\begin{gather*}
|\gg(\bj) - \eta_\bj| \leq 2^{4d+2} \beta \left(\frac{\ln n}{n}\right)^k.
\end{gather*}
\end{prop}

\begin{proof}
Fix $\bj \in [n]_k$. We will apply Proposition \ref{fknprop1} to the function $\poch_\bj g \boolf V_d$. Let us check the assumption that $\totinf(\poch_\bj g) < \frac{1}{2^{5d}}$. Indeed, we have $\sqrt{n} \geq \ln n$ and $\beta < \frac{1}{2^{5d}}$, so
\begin{equation}
\totinf(\poch_\bj g) = \sum_{i \in [n]} \infl_i(\poch_\bj g) = \!\! \sum_{\substack{\ \bi \in [n]_{k+1}\!\colon\\ \bj \subseteq \bi}}\!\! \infl_\bi(g) \leq \!\! \sum_{\substack{\ \bi \in [n]_{k+1}\!\colon\\ \bj \subseteq \bi}} \beta \left(\frac{\ln n}{n}\right)^{k+1} \leq \beta \frac{(\ln n)^{k+1}}{n^k} < \frac{1}{2^{5d}}. \label{intermedbound}
\end{equation}
Therefore, we get $\eta_\bj \in \imag(\poch_\bj g) \subseteq V_d$ satisfying
\begin{gather}
\PP(\poch_\bj g \neq \eta_\bj) \leq \frac{2^{2d-1}\totinf(\poch_\bj g)}{\ln\left(\frac{1}{2^{5d}\totinf(\poch_\bj g)}\right)} \leq 2^{2d-1} \beta \left(\frac{\ln n}{n}\right)^k \cdot \frac{\ln n}{\ln\left(\frac{n^k}{(\ln n)^{k+1}}\right)} \leq 2^{2d+1}\beta \left(\frac{\ln n}{n}\right)^k, \label{pochneq}
\end{gather}
where we used an intermediate bound from (\ref{intermedbound}) and the inequality $2^{5d}\beta < 1$.

Now observe that
\begin{align*}
\gg(\bj) = \EX [\poch_\bj g] &= \eta_\bj \PP(\poch_\bj g = \eta_\bj) + \EX [(\poch_\bj g) \IND_{\{\poch_\bj g \neq \eta_\bj\}}] \\
&= \eta_\bj (1 - \PP(\poch_\bj g \neq \eta_\bj)) + \EX [(\poch_\bj g) \IND_{\{\poch_\bj g \neq \eta_\bj\}}].
\end{align*}
Thus, as $|\eta_\bj| \leq 2^{2d}$ and $\|\poch_\bj g\|_{\infty} := \max_{x \in \{-1,1\}^n} |\poch_\bj(x)| \leq 2^{2d}$, we get
\begin{equation*}
|\gg(\bj) - \eta_\bj| \leq (|\eta_\bj| + \|\poch_\bj g\|_{\infty}) \PP(\poch_\bj g \neq \eta_\bj) \leq 2^{4d+2} \beta \left(\frac{\ln n}{n}\right)^k. \qedhere
\end{equation*}
\end{proof}

\subsection{Coefficients of a $d$-degree function}

In this section we prove that the Fourier--Walsh coefficients of a $d$-degree Boolean function lie in the $\ZZ / 2^{d-1}$ lattice. This result in print has appeared in \cite{KMS12} [Proposition 3.1], but we include our rather short proof to keep the work self-contained.

\begin{prop} \label{wsppocho}
Let $d \geq 1$ be an integer and $g \boolf \{-1,1\}$ be a $d$-degree function. Then for every $\bi \subset [n]$ we have $\gg(\bi) \in \ZZ / 2^{d-1}$. Consequently, at most $2^{2d-2}$ of the coefficients $\gg(\bi)$ are non-zero.
\end{prop}

\begin{proof}
We have $\gg(\bi) = 0$ for $\bi \in [n]_{>d}$, so it is enough to resolve the $\bi \in [n]_{\leq d}$ case.

Assume, for a contradiction, that there is $\bi \in [n]_{\leq d}$ with $\gg(\bi) \notin \ZZ / 2^{d-1}$, and choose such a set $\bi$ of maximal cardinality. We have
\begin{equation} \label{eq:ddegcoef}
\poch_\bi g = \gg(\bi) + \!\!\!\sum_{\substack{\ \bj \in [n]_{\leq d}\!\colon \\ \bi \subsetneq \bj}}\!\! \gg(\bj) \chi_{\bj \setminus \bi}.
\end{equation}
Since $g$ is Boolean, by \cref{der_int_mult}, $\imag(\poch_\bi g) \subset \ZZ / 2^{|\bi|-1} \subset \ZZ  / 2^{d-1}$. Further, by the choice of~$\bi$, we have $\gg(\bj) \in \ZZ / 2^{d-1}$ for all $\bj \in [n]_{\leq d}$ such that $\bj \supsetneq \bi$. Hence, all of the terms but one in \eqref{eq:ddegcoef} are integer multiples of $1/2^{d-1}$, which yields $\gg(\bi) \in \ZZ / 2^{d-1}$, a contradiction.

As by Parseval's theorem $\sum_{\bi \subset [n]} \gg(\bi)^2 = 1$, clearly at most $2^{2d-2}$ of the coefficients $\gg(\bi)$ are non-zero.
\end{proof}

\section{Proof of \cref{FKN-type theorem}} \label{sec:fkn}

We are now ready to prove \cref{FKN-type theorem}. Let us sketch the plan of the proof.

First, thanks to \cref{Main Theorem} and the Kindler--Safra \cref{kindlersafra}, the assumption \eqref{maxinfassump} of small $(d+1)$-degree influences gives a candidate $d$-degree Boolean function $g$ approximating~$f$. All we need to prove is that this function $g$ is sufficiently close to $f$ in the sense of the bounds \eqref{statementtwierdzenia}. We show this by a strengthened induction, where in parallel to proving the bound we construct a scheme of (quasi-Boolean) functions with small influences which have good approximating properties owing \cref{fknprop4,fknprop5}. 

We will prove a slightly (and equivalently) rephrased statement of \cref{FKN-type theorem}.

\begin{theorem} \label{fkntwierdzenie}
Let $n > d \geq 1$. There are positive constants $C_3$, $C_4$ depending only on $d$ such that if $f \boolf \{-1,1\}$ satisfies 
\begin{equation} \label{maxinfassump}
\maxinf_{d+1}(f) \leq \alpha \left(\frac{\ln n}{n}\right)^{d+1}
\end{equation}
for some $\alpha \in (0,C_3)$, then there exists a $d$-degree function $g \boolf \{-1,1\}$ such that for every $\bj \in [n]_{\leq d}$ we have
\begin{equation}\label{statementtwierdzenia}
|\ff(\bj) - \hat{g}(\bj)| \leq C_4 \alpha \left(\frac{\ln n}{n}\right)^{|\bj|}.
\end{equation}
\end{theorem}

\begin{proof}
Observe that the bound \eqref{maxinfassump} and \cref{Main Theorem} yield $\bW^{\geq d+1}(f) \leq 10 \alpha$. Hence, the Kindler--Safra theorem (\cref{kindlersafra}) implies the existence of a $d$-degree $g\boolf \{-1,1\}$ such that 
\begin{equation} \label{wniosekks}
 \|f-g\|_2^2 \leq 10C_{KS} \alpha.
\end{equation}
We will prove that the $g$ obtained satisfies the bound (\ref{statementtwierdzenia}) for~$C_3 = \min\{2^{-10d^2}, (2^{2d}10C_{KS})^{-1}\}$ and $C_4 = \max\{2^{8d^2}, 10C_{KS}\}$. 

First, let us note that the bound (\ref{statementtwierdzenia}) holds for $\bj = \emp$. Indeed,
\begin{equation}\label{empcase}
|\ff(\emp) - \gg(\emp)| = | \EX [f-g] | \leq \EX [|f-g|] \leq \EX [(f-g)^2] = \|f-g\|_2^2 \leq 10C_{KS} \alpha \leq C_4\alpha.
\end{equation}

Define $f_1, \ldots, f_d \boolf \RR$ by $f_d := f$ and
\begin{gather}
f_k \ := \  f - \sum_{\substack{\ \bi \subset [n]\!\colon \\ k < |\bi| \leq d}} \gg(\bi) \chi_\bi \ = \ f_{k+1} - \sum_{\bi \in [n]_{k+1}} \gg(\bi) \chi_\bi. \label{rekurencja}
\end{gather}
\cref{wsppocho} yields $\gg(\bi) \in \ZZ / 2^{d-1}$ for every $\bi \in [n]_{\leq d}$. Fix $\bj \in [n]_{\leq d}$. By \cref{der_int_mult} we have $\imag(\poch_\bj f) \subset \ZZ/2^{d-1}$. In turn, due to the Fourier--Walsh expansion of $\poch_\bj$, we have $\imag(\poch_\bj \chi_\bi) \subset \{-1,0,1\}$. Therefore, \eqref{rekurencja} yields $\imag(\poch_\bj f_k) \subset \ZZ / 2^{d-1}$. 

On the other hand, the second part of \cref{wsppocho} implies the pointwise bound
\[|f_k| \leq |f| + \sum_{\substack{\ \bi \subset [n]\!\colon \\ k < |\bi| \leq d}} |\gg(\bi)| \leq 1 + 2^{2d-2} < 2^{2d}.\]
Hence,
\begin{equation} \label{imgpoch}
\imag(\poch_\bj f_k) \subset [-2^{2d}, 2^{2d}] \cap \ZZ / 2^{d-1} = V_d   \text{ for every } k \in [d] \text{ and } \bj \in [n]_{\leq d}.
\end{equation}

Let $\alpha_d = \alpha$ and recursively define $\alpha_k = 2^{8d}\alpha_{k+1}$ for $k \in [d-1]$. Due to the choice of $C_3$, we have
\begin{equation*}
\alpha_k = 2^{8d(d-k)}\alpha \ \quad \text{and} \quad \ \alpha = \alpha_d \leq \alpha_{d-1} \leq \ldots \leq \alpha_1 \leq \frac{1}{2^{10d}} < 1.
\end{equation*}
We claim that for every $k\in [d]$ the following conditions are satisfied:
\begin{enumerate}[label=(\textbf{C\arabic*}]
\item\!\!\!$_k$) $\maxinf_{k+1}(f_k) \leq \alpha_k \left(\frac{\ln n}{n}\right)^{k+1}$;
\item\!\!\!$_k$) for every $\bj \in [n]_{k}$ we have 
\[\infl_\bj(f_k) - \ff_k(\bj)^2 \leq 2^{2d}\alpha_k \left(\frac{\ln n}{n}\right)^k;\]
\item\!\!\!$_k$) for every $\bj \in [n]_{k}$ we have
\[|\ff_k(\bj) - \gg(\bj)| =  |\ff(\bj) - \gg(\bj)| \leq 2^{4d+2}\alpha_k \left(\frac{\ln n}{n}\right)^k.\]
\end{enumerate}

Observe that $2^{4d+2}\alpha_k \leq 2^{8d^2}\alpha$, so that (\ccc) and \eqref{empcase} imply the theorem with the chosen constant $C_4$.

In order to prove the claim, we run a downwards induction on $k$. We perform it in two steps. First, we prove that for every $k \in [d]$ the condition (\c$_k$) implies (\cc$_k$) and (\ccc$_k$). Second, we prove that conditions (\cc$_k$) and (\ccc$_k$) imply (\c$_{k-1}$).

The base condition (\c$_d$) is simply the assumption (\ref{maxinfassump}). 

Let $\bj \in [n]_k$ and assume (\c$_k$) holds.

\noindent(\cc$_k$): Because of (\c$_k$) and (\ref{imgpoch}), we are allowed to apply \cref{fknprop4} to the function $f_k$ and $\beta = \alpha_k$, which yields the condition.

\noindent(\ccc$_k$): Because of (\c$_k$) and (\ref{imgpoch}), we are allowed to apply \cref{fknprop5} to the function $f_k$ and $\beta = \alpha_k$. As a result, we obtain $\eta_\bj \in V_d$ such that 
\begin{equation} \label{wniosekprop5}
|\ff_k(\bj) - \eta_\bj| \leq 2^{4d+2}\alpha_k \left(\frac{\ln n}{n}\right)^k.
\end{equation}
Since $|\bj| \leq k$, the definition of $f_k$ yields $\ff_k(\bj) = \ff(\bj)$. Thus, it is enough to show that $\gg(\bj) = \eta_\bj$.

Let us observe that for every $\bi \subset [n]$ we have
\begin{equation} \label{fgblisko}
|\ff(\bi) - \gg(\bi)| = \sqrt{(\ff(\bi) - \gg(\bi))^2} \leq \|f-g\|_2 \stackrel{\eqref{wniosekks}}{\leq} \sqrt{10C_{KS} \alpha} < \frac{1}{2^d},
\end{equation}
so by \eqref{wniosekprop5} and inequality $2^{4d+2} \alpha_k < \frac{1}{2^d}$, 
\begin{equation} \label{etagblisko}
|\eta_\bj - \gg(\bj)| \leq |\ff(\bj) - \gg(\bj)| + |\ff(\bj) - \eta_\bj| < \frac{1}{2^{d-1}}. 
\end{equation}
Recall that $\eta_\bj \in V_d$ and that $\gg(\bj)$ is a coefficient of a $d$-degree Boolean function. Hence, both $\eta_\bj$ and $\gg(\bj)$ are integer multiples of $\frac{1}{2^{d-1}}$. Thus, (\ref{etagblisko}) implies $\gg(\bj) = \eta_\bj$, and the condition follows.

Turning to the second step, let us assume that (\cc$_k$) and (\ccc$_k$) hold for some $k \in \{2, \ldots, d\}$.

\noindent(\c$_{k-1}$): Let $\bi \in [n]_k$. By the definition of $f_{k-1}$ and conditions (\cc$_k$), (\ccc$_k$), we get
\begin{align*}
\infl_\bi(f_{k-1}) &= \infl_\bi(f_k) - \ff_k(\bi)^2 + (\ff_k(\bi) - \gg(\bi))^2 \\
&\leq 2^{2d} \alpha_k \left(\frac{\ln n}{n}\right)^k + 2^{8d+4}\alpha_k^2\left(\frac{\ln n}{n}\right)^{2k} \leq  \alpha_{k-1}\left(\frac{\ln n}{n}\right)^k,
\end{align*}
as $\alpha_k \leq \tfrac{1}{2^{10d}}$. Since $\bi \in [n]_k$ was arbitrary, the condition follows. 
\end{proof}

\begin{remark}
It is possible to avoid the use of the Kindler--Safra theorem. Our original argument showed directly that the function $\sum \eta_\bi \chi_\bi$ is the desired $d$-degree Boolean $g$. Such an approach, although elementary, is considerably longer.
\end{remark}

\begin{remark}
A non-trivial example of $f$ satisfying the assumption (\ref{maxinfassump}) of small $d$-degree influences is a slightly modified version of the hypertribe function described in \cref{sec:ex}. The nuance to overcome in this case is making the parameter $\alpha$ smaller than $C_3$. In order to guarantee this, we may increase the size of the tribes by a constant (depending on $d$). After this modification, the bounds obtained in \cref{properties} will be still satisfied for sufficiently large $n$, and following the proof of \cref{prop111}, one can see that $\alpha$ will be suitably small. In this case, for sufficiently large $d$ and $n$, \cref{prop111} and \cref{infl_lower_bd} yield that the approximating function $g$ is the constant $-1$. 
\end{remark}


We close with the proof of \cref{ourcorr}.

\begin{proof}[Proof of \cref{ourcorr}] We show that the function $g$ provided by \cref{FKN-type theorem} satisfies the bound. Assume $\bi \in [n]_{\leq d}$ is such that $|\ff(\bi)| < \frac{1}{2^d}$. Since $g$ is $d$-degree, \cref{wsppocho} yields that $\gg(\bi)$ is an integer multiple of $\frac{1}{2^{d-1}}$. Therefore, appealing to the above proof, \eqref{fgblisko} implies $\gg(\bi) = 0$. 

Thus, for all $\bi \in [n]_{< l}$ we have $\gg(\bi) = 0$, so by Plancherel's theorem,
\begin{equation*}
\|f-g\|_2^2 = \|f\|_2^2 - \|g\|_2^2 + 2\langle g-f, g \rangle = 2\langle g-f, g \rangle = 2 \sum_{\substack{\ \bi \subseteq [n]\!\colon\\ l \leq |\bi| \leq d}} \!\! \big(\gg(\bi) - \ff(\bi)\big) \cdot \gg(\bi).
\end{equation*}
Appealing again to \cref{wsppocho}, we know that at most $2^{2d-2}$ of the coefficients $\gg(\bi)$ are non-zero, so using that $|\gg(\bi)| \leq 1$ and by \cref{FKN-type theorem},
\begin{equation}
\|f-g\|_2^2 \leq 2 \sum_{\substack{\ \bi \subseteq [n]\!\colon\\ l \leq |\bi| \leq d}} \! \big|\hat{g}(\bi) - \ff(\bi)\big| \cdot |\hat{g}(\bi)| \leq 2^{2d-1} \cdot C_4 \alpha \left(\frac{\ln n}{n}\right)^l.
\end{equation}
The claim follows with $C_5 = 2^{2d-1}C_4.$
\end{proof}

\bibliographystyle{amsplain}

\end{document}